\documentclass[12pt,a4]{article}
\usepackage{amsthm,amsmath,amssymb,graphicx,hyperref,color,mathabx,cite}

\newtheorem{definition}{Definition}
\newtheorem{lemma}{Lemma}
\newtheorem{proposition}{Proposition}
\newtheorem{theorem}{Theorem}
\newtheorem{remark}{Remark}

\newtheorem{corollary}{Corollary}

\newcommand{\R}{\mathbb{R}}

\newcommand{\orbit}{\mathcal{O}}
\newcommand{\fcham}{\mathcal{C}}
\renewcommand{\phi}{\varphi}
\newcommand{\transp}{^{\mathrm{T}}}

\newcommand{\bd}{\operatorname{bd}}
\newcommand{\epi}{\operatorname{epi}}
\newcommand{\Proj}{\mathcal{P}}

\DeclareMathOperator*{\Argmin}{Arg\, min}

\DeclareMathOperator{\Sph}{\mathcal{S}}
\DeclareMathOperator{\co}{co}
\DeclareMathOperator{\dist}{d}
\DeclareMathOperator{\cone}{cone}
\DeclareMathOperator{\Perm}{\mathbf{P}}

\title{Orbital Geometry in Optimisation}
\author{Andrew Eberhard\thanks{School of Mathematical and Geospatial Sciences, RMIT University, andy.eb@rmit.edu.au} and Vera Roshchina\thanks{School of Mathematical and Geospatial Sciences, RMIT University, vera.roshchina@rmit.edu.au}}

\begin{document}

\maketitle

\begin{abstract}
We discuss the use of group symmetries in optimisation, in particular with respect to the structure of subdifferential and projection operators. This allows us to generalise a classic result of Adrian Lewis regarding the characterisation of the subdifferential of a permutation invariant convex function to the characterisation of the proximal subdifferential of a Schur convex function that is invariant with respect to a finite reflection group. We are also able to simplify and generalise results on projections onto symmetric sets, in particular, we study projections on sparsity constraints used in sparse signal recovery and compressed sensing.

\end{abstract}

\section{Introduction}
The in depth study of the interplay between group symmetry and nonsmooth optimisation was pioneered by Adrian Lewis, who obtained a range of results in subdifferential calculus under group transformations to problems in eigenvalue optimisation. The original motivation for this direction of research comes from the observation that eigenvalue problems possess inherent symmetry: matrix eigenvalues do not have an intrinsic ordering, and hence such problems have a certain ambiguity. The symmetry that eigenvalues enjoy is permutation invariance, and can be studied via the properties of permutation group and its linear representations, so that the problem can be reduced to working on the quotient space obtained by glueing together the orbits under the action of permutation group. Our goal is to put these results into a more general framework of finite reflection groups, which helps to provide a clean intuition, a greater clarity and simplicity of structure so as to obtain generalisations via simplified proofs, and to  explain some other results from more applied areas such as compresses sensing. 

The aforementioned work by Adrian Lewis \cite{Lewis1996SIAM,Lewis1996,Lewis99} and its generalisations, notably  \cite{Lewis2000, HillGroupsThesis,HillPaper,ArisGeneralisedLewis}, however bold and ground breaking, are not the first nor the only works relating symmetry with nonsmooth optimisation. The 1979 work by Palais \cite{Palais} is foundational in establishing major results in the area, such as the principle of symmetric criticality that allows to deduce the optimality of the unconstrained solution to an optimisation problem from the optimality on an invariant subspace. Recent works \cite{KristalyVargaVarga2007,KobayashiOtani2004} explore these ideas further. We refer the reader to the recent work \cite{BorweinZhuSymmetry} that contains a comprehensive survey of results on the topic. When it comes to applications, symmetry is heavily explored in a range of areas. For instance, in linear and integer programming it is utilised to reduce computational time \cite{liberti-chapter,margot-chapter}. Recently there has been an explosion of research in conic optimisation, and particularly in semidefinite programming that has a strong algebraic flavour and utilise group symmetries to construct new optimisation techniques. To give just a few recent examples, group symmetry is utilised in the seminal work on kissing number problems \cite{Kissing} (see \cite{MathPacking} for a modern treatment of the underlying theory). The role of symmetries in polyhedral lifting problems is studied in \cite{GouveiaLifts,Saunderson}. A range of works in conic optimisation utilise symmetry to reduce the dimension of the problem, e.g. see \cite{SDPVallentin,JuanVeraSym}. Our focus will be to concentrate on the structural insights
into subgradients and projections that can be obtained by group structure. As 
projections are central to basic constructions in nonsmooth analysis (such 
as normal cones to sets) these two topics go hand in glove enabling a path
way to a more general theory.

In this paper we focus on finite reflection groups, as this setting allows enough generality to explain the underlying intuition, yet helps avoid cumbersome technicalities. Many of the results of this paper can be framed within the  context of 
a normal decomposition system as defined in \cite{Lewis1996SIAM}. In \cite{HillGroupsThesis} some related results may be  found and are developed within
the context of reduced Eaton Triples. As noted there and shown in \cite{Niezgoda1998} a reduced Eaton triple corresponds to a finite reflection group
and it is noted in \cite{HillGroupsThesis} that reduced Eaton triples ``almost'' correspond
to a normal decomposition sub-system as defined by Lewis in \cite{Lewis1996SIAM}.
The only possible difference in properties being the attainment of the maximum of 
the inner product taken over an orbit, which may not in general be attained for Eaton Triples. 
In this paper we require this to be achieved 
by the unique elements corresponding to the intersection of the orbit with the fundamental Weyl chamber. 
Consequently we restrict attention to the very natural class of finite reflection groups acting on
a finite dimensional inner product space, departing from the context of 
\cite{HillGroupsThesis,HillPaper}. 

We begin with the basic notions and results related to finite reflection groups in Section~\ref{sec:groups}, then in Section~\ref{sec:subdifferential} we obtain several technical results concerning the subdifferentials of convex functions invariant under group actions. Section~\ref{sec:projections} is devoted to the study of projections on the convex sets where we obtain a characterisation of projections in terms of group stabilisers. In Section~\ref{sec:normals} we obtain a range of results on the structure of proximal normals to symmetric sets and study the proximal subdifferential of Schur convex functions. We finish with revisiting projections on sparsity constraints in the framework of compressed sensing in Section~\ref{sec:compressed}. 

Throughout the paper, we let $\langle x,y \rangle = x\transp y$, and $\|x\| = \|x\|_2 = \sqrt{x\transp x}$. By $\Sph_{n-1}$ we denote the unit sphere in $\R^n$. By 
$\co C$ and $\cone C$ we denote respectively the convex and conic hulls of set $C$  and the extended real number system $\R_{+\infty} := \R \cup \{+\infty \}$. 

\section{Finite reflection groups}\label{sec:groups}

Any finite Coxeter group has a linear representation generated by a family of reflections about a hyperplane in $\R^n$. We are interested in studying the properties of sets and functions invariant under the action of this group. In what follows we slightly abuse the notation, and often omit the word `action' for the sake of brevity. 
Observe that a finite reflection group $G$ is a subgroup of the orthogonal group $O(n)$, and even though $G$ is generated by reflections, not all transformations in $G$ are reflections themselves. For example, the symmetry group of the square $\mathrm{Dih}_4$ has the representation as a finite reflection group generated by two reflections about mirrors positioned at a $\pi/4$ angle. It is not difficult to observe that this group consists of the identity, the four reflections, and three nontrivial rotations (see Fig.~\ref{fig:square_symmetry}).
\begin{figure}[ht]
	\centering
	\includegraphics[scale=0.7]{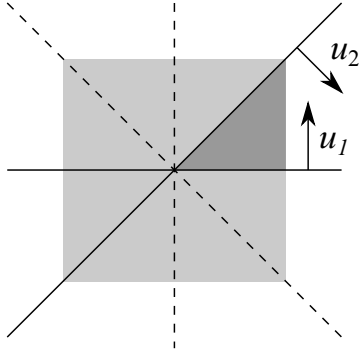}
	\caption{The mirrors of the square symmetry group $\mathrm{Dih}_4$}
	\label{fig:square_symmetry}
\end{figure}

A mirror is a hyperplane in $\R^n$ with associated unit normal $u\in \Sph_{n-1}$; the reflection is represented by the Householder transformation $H_u = I-2uu\transp$, which fixes this hyperplane. We denote the relevant group action by $h_u$, i.e. $h_u x = H_ux$. 


Observe that the mirrors in Fig~\ref{fig:square_symmetry} split the space into eight wedges, each is a transformation of any other obtained by consecutive reflections via adjacent sides. Moreover, we can pick a `generating wedge' in an arbitrary way. Higher dimensional representations of finite reflection groups have more complicated structures, but this core observation is still true: the underlying space can always be subdivided into chambers bounded by mirrors (hyperplanes) that play the same role as wedges in the planar groups. We can choose any one of the closed chambers and call it {\em fundamental}. A core result in the theory of finite groups is that any fundamental chamber is also a fundamental domain (see Lemma~\ref{lem:fchamunique}), i.e. it contains a unique representative from each of the orbits (the sets formed by all images of a given point under all group actions). Before we state this and other results related to the structure of the finite reflection groups, we formalise the notation.

Recall that for a group $G$ acting on $\R^n$, and a point $x\in \R^n$, the {\em orbit} of $x$ is the set of all its images under $G$:
$$
\orbit(x) = \{y\in \R^n\,|\, \exists g\in G,\, y = gx\}.
$$
A set $\Phi \subset \Sph_{n-1}$ is a {\em root system} of a finite reflection group $G$ if 
$$
\Phi = \{u\in \Sph_{n-1}\,|\, h_u \in G\},
$$
i.e. it consists of all normals (positive and negative) to all reflection hyperplanes in the group.

A set $U \subset \Phi$ is a positive root system if there exists a linear mapping $f(x):
\R^n \to \R$ with $f(u)>0$ for all $u\in U$, and $f(u)<0$ for all $u\in \Phi\setminus U$. Observe that it is always possible to construct a positive root system, since the group $G$ is finite. 

Given a positive root system $U$ of a group $G$ the closed fundamental chamber $\fcham$ of $G$ is the dual cone of the positive root system 
$$
\fcham = (\cone U)^* = \{x\in \R^n\, |\, \langle x,u\rangle \geq 0 \quad \forall u \in U\}.
$$

Observe that this definition does not contradict the arbitrariness of choice for the fundamental chamber: it is always possible to choose a positive root system for the given selected chamber. The positive root system of the dihedral group $\mathrm{Dih}_4$ shown in Fig.~\ref{fig:square_symmetry} consists of two vectors $u_1$ and $u_2$ orthogonal to the generating mirrors, and the fundamental chamber is the darker wedge that corresponds to the dual cone of the pair $(u_1,u_2)$.

Our development heavily relies on the following well-known result (for the proof see \cite[Theorem (a), page 22]{HumphreysCoxeterGroups}).
\begin{lemma}\label{lem:fchamunique} For any finite reflection group $G$ and any $x\in \R^n$ the set 
$$
\orbit (x)\cap \fcham,
$$
where $\orbit (x)$ is the orbit of $x$, is a singleton. 
\end{lemma}

By Lemma~\ref{lem:fchamunique} we can map any point $x\in \R^n$ to its unique intersection $\widecheck x$ with the fundamental chamber, i.e. 
$$
\{\widecheck x\} = C\cap \orbit(x).
$$

We will need one more technical result that concerns group stabilisers. Recall that a {\em stabiliser} (isotropy group, point wise centraliser) $C_G(x)$ of $x\in \R^n$ is the subset of $G$ that fixes $x$, i.e.
$$
C_{G}(x)= \{ h \in G \mid hx=x \}.
$$ 

\begin{lemma}\label{lem:subgrouproots} Let $G$ be a finite reflection group with a positive root system $U$, and assume $C_G(x)$ is a stabiliser of $x\in \R^n$. Then $C_G(x)$ is generated by those reflections that $C_G(x)$ contains. 
\end{lemma} 
\begin{proof}
Follows directly from Theorem 12.6 in \cite{Mirrors2010}.
\end{proof}

\section{Subdifferential of a symmetric convex function}\label{sec:subdifferential}

The purpose of this section is to embed the results of Adrian Lewis \cite{Lewis99} in
the framework of finite reflection groups, in a transparent way. Our results are in the same spirit as those of \cite{HillGroupsThesis,HillPaper}. In particular, we rephrase the following result (\cite{Lewis99,BorweinSymm}) in terms of finite reflection groups: if a function $f:\R^n\to \R$ is invariant with respect to the permutation of coordinates, then $y\in \partial f(x)$ if and only if 
\begin{equation}\label{eq:main01}
y^\downarrow \in \partial f(x^\downarrow) \quad \text{and} \quad \langle x^\downarrow, y^\downarrow \rangle = \langle x,y\rangle,  
\end{equation}
where $\langle x,y\rangle = x\transp y$ is the scalar product, and $x^\downarrow$ denotes the nonincreasing reordering of coordinates.

In our developments the permutation group is replaced by the finite reflection group, and the role of reordering of coordinates is played by the operation $\widecheck x$ defined earlier as the intersection of the orbit of the point $x$ with the fundamental chamber. We say that a function $f:\R^n\to \R$ is invariant under a reflection group $G$ if for every $g\in G$ and every $x\in \R^n$ we have $f(gx) = f(x)$. The main goal of this section is to prove the following generalised result.

\begin{theorem}\label{thm:main:1} Let $f:\R^n\to \R$ be a convex function invariant under a finite reflection group $G$.
	Then $y\in \partial f(x)$ if and only if 
	\begin{equation}\label{eq:main}
	\widecheck  y \in \partial f(\widecheck  x) \quad \text{and} \quad \langle \widecheck  x, \widecheck  y\rangle = \langle x,y\rangle,
	\end{equation}
	where by $\widecheck  x$ (resp.  $\widecheck  y$) we denote the unique point that belongs to  the intersection of the orbit $\orbit_G(x)$ of $x$ (resp. $\orbit_G(y)$ of $y$) with the fundamental chamber $\fcham$ of $G$.
\end{theorem}

Before we go on with the proof we state some well known facts and technical results.

Recall that given a convex function $f:\R^n \to \R$ its Moreau-Rockafellar {\em subdifferential} $\partial f(x)$ (see \cite{Rockafellar1970}) is the set of such $v\in \R^n$ (called subgradients) that 
$$
f(y)\geq f(x) +\langle v, y-x\rangle \quad \forall \, y \in \R^n.
$$

The next result is well-known and follows directly from the subdifferential chain rule. We provide the proof for convenience.
\begin{lemma}\label{lem:subdiffchain} Let $x\in \R^n$, and assume $f:\R^n \to \R^n$ is convex and invariant under a reflection group $G$. Then for any $y\in \partial f(x)$ and any $g\in G$
\begin{equation}\label{eq:003}
\partial f(x) = g^* \partial f(gx) \quad \text{ and so }\quad  g\partial f(x) = f (gx), 
\end{equation}
where $g^*$ is the adjoint linear operator for $g$. Moreover, if $x$ is such that $x = h x $ for some $h\in G$, then
\begin{equation}\label{eq:004}
\partial f(x) = h \partial f(x).
\end{equation}
\end{lemma}
\begin{proof}
First observe that every mapping $g\in G$ is surjective, hence, we can apply the subdifferential chain rule  (see \cite[Theorem 23.9]{Rockafellar1970}):
$$
\partial (f\circ g)(x) = g^* \partial f(gx).
$$
Since in our case $f$ is $G$-invariant, we have $(f\circ g ) (x) = f(x)$, hence, $\partial (f\circ g )(x) = \partial f(x)$, and we get \eqref{eq:003}. The relation \eqref{eq:004} follows from \eqref{eq:003} by substituting $x=gx$ in the right hand side.
\end{proof}

The combination of group structure and convexity allows for more precise characterisations of the Moreau-Rockafellar subdifferential in terms of the stabiliser given in Lemma~\ref{lem:subdiffstructure}. We need the following technical result first.

\begin{lemma}\label{lem:posoutside} Let $f:\R^n\to \R$ be invariant under a finite reflection group $G$, and let $U$ be a positive root system of $G$. For  $x \in \fcham$ let 
$$
V(x) := \{v\in U\, |\, h_v \in C_{G}(x)\},
$$ 
where $C_G(x)$ is the stabiliser of $x$. Then for any $y\in \partial f(x)$ and $u \in U\setminus V$ we have $\langle y,u\rangle \geq 0$.
\end{lemma}
\begin{proof} Fix $x\in \fcham$ and pick an arbitrary $y \in \partial f(x) $ and $u\in U\setminus V$. From the convexity and invariance of $f$ we have for the reflection $h_u$ represented by the Householder transformation $H_u = I - 2 u u\transp$:
$$
f(x) = \frac{1}{2}f(x)+\frac{1}{2}f(h_u x) \geq f\left(\frac{x+h_u x}{2}\right) = f(x-\langle u,x\rangle u)\geq f(x) -\langle u,x\rangle \langle u, y\rangle,
$$ where the last inequality follows from the definition of the subdifferential.
Hence,
$$
\langle u,x\rangle \langle u, y\rangle\geq 0.
$$ 
Since $x \in \fcham$, and $\langle x,u\rangle \neq 0$ (recall that $h_u$ is {\em not} in the stabiliser), we have $\langle u,x\rangle>0$, and therefore $\langle u,y\rangle \geq 0$.
\end{proof}

\begin{lemma}\label{lem:subdiffstructure} Let $f:\R^n\to \R$ be a convex function invariant under a finite reflection group action $G$, and assume $x\in \fcham$. Then for any $g\in G$
$$
\partial f(gx) = g(C_G(x)(\partial f(x)\cap \fcham)),
$$
where $C_G(x)$ is the stabiliser of $x$.
\end{lemma}
\begin{proof} By Lemma~\ref{lem:subdiffchain} we have $\partial f(gx) =  g\partial f(x) $, hence, it is sufficient to show that
\begin{equation}\label{eq:004}
\partial f(x) = C_G(x)(\partial f(x)\cap \fcham) \qquad \forall \, x\in \fcham.
\end{equation}

Let $x\in \fcham$. If $y\in \partial f(x)$, then by Lemma~\ref{lem:subdiffchain} we have  $h y\in \partial f(x)$ for all $h \in C_G(x)$, hence, 
$$
C_G(x)(\partial f(x)\cap \fcham) \subset \partial f(x).
$$

To show the reverse inclusion, let $y\in \partial f(x)$. From the positive root system $U$ choose the subsystem $V = \{v\in U\,|\, h_v \in C_G( x)\}$. By Lemma~\ref{lem:subgrouproots}, $V$ is the positive root system of $C_G(x)$ defined by the same linear mapping as $U$. By Lemma~\ref{lem:fchamunique} the intersection of each orbit with the fundamental chamber is unique, hence, there exists $h\in C_G(x)$ such that $\langle hy, v\rangle \geq 0$ for all $v\in V$ i.e $hy \in \orbit (y)\cap (\cone V)^{\ast}$. For all $u\in U\setminus V $ Lemma~\ref{lem:posoutside} yields  $\langle hy, u\rangle \geq 0 $. Observe that also by Lemma~\ref{lem:subdiffchain} we have $hy\in \partial f(x)$. We hence conclude that  
$hy \in \fcham\cap \partial f(x)$, and hence $y = h^* (h y) \in  C_G(x)(\partial f(x)\cap \fcham)$. By the arbitrariness of $y$ this yields the desired inclusion
$$
\partial f(x) \subset C_G(x)(\partial f(x)\cap \fcham).
$$
\end{proof}

We are now ready to present the proof of the main theorem.

\noindent{\em Proof of Theorem~\ref{thm:main:1}.} 
Let $y\in \partial f(x)$. There exists $g \in G$ such that $x= g\widecheck  x$. From Lemma~\ref{lem:subdiffstructure} 
$$
y \in  g (C_G(\widecheck  x)(\partial f(\widecheck  x) \cap \fcham)),
$$
therefore, there exists $h\in  C_G(\widecheck  x) $ and $y'\in \partial f(\widecheck  x) \cap \fcham$ such that $y = gh y'$. Thus $y' \in \orbit (y) \cap \fcham$. By Lemma~\ref{lem:fchamunique} such $y'$ is unique and must coincide with $\widecheck  y$. Hence $y = g h y' = g h \widecheck  y$. Since $h$ fixes $\widecheck  x$, we also have $x = g\widecheck  x = g h \widecheck  x$. Applying Lemma~\ref{lem:subdiffchain}, we have 
$$
\widecheck  y = h^* g^* y  \in h^* g^* \partial f(x) = h^*g^*\partial f(g h \widecheck  x) = \partial f(\widecheck  x). 
$$
Finally, 
$$
\langle x, y \rangle = \langle g h \widecheck  x, g h \widecheck  y \rangle = \langle \widecheck  x, \widecheck  y\rangle.
$$

Now assume that \eqref{eq:main} holds for some $x, y \in \R^n$. 
Then there exists $g \in G$ such that  $\widecheck{x} = g x$ and Lemma~\ref{lem:subdiffchain} gives 
$$
\widecheck{y} \in \partial f (\widecheck{x}) = g \partial f(x).
$$
Thus there exists $y^{\prime} \in \partial f (x)$ with $\widecheck{y} = g y^{\prime}$, and $y^{\prime} \in  \partial f(x)$. As $y^{\prime} \in \orbit (\widecheck{y}) = \orbit (y)$, there is a 
$k \in G$ with $ky = y^{\prime}$. As $y^{\prime} \in \partial f(x)$ via the subgradient inequality (and invariance of $f$) we have  for all $z$ that 
\begin{align*}
f(z) - f(x) 
& = f(kz) - f(x) \geq  \langle  y^{\prime} , kz - x \rangle \\
& = \langle k y , kz \rangle - \langle y^{\prime} , x \rangle  
  = \langle y , z \rangle - \langle g^{\ast} \widecheck y , g^{\ast} \widecheck x \rangle  \\
 & = \langle y , z \rangle - \langle y , x \rangle 
  = \langle y , z - x \rangle,
\end{align*}
hence $y \in \partial f (x)$.\qed 

\section{Projections onto symmetric sets}\label{sec:projections}

It appears that projections onto permutation invariant sets possess even more structure than the Moreau-Rockafellar subdifferential, moreover, several results can be stated in terms of sets and functions more general than convex. We begin with reminding several known technical results. The following result was first obtained in \cite{EatonPerlman1977}.
\begin{lemma}\label{lem:3}
	Suppose $y, x \in \R^n$ then for all $g \in G$ we have 
	\[
	\langle g \widecheck y  , \widecheck x  \rangle \leq \langle \widecheck y  , \widecheck x  \rangle. 
	\]
	Moreover we have equality if and only if there exists $q \in C_G(\widecheck x)$ such that
	$q (g \widecheck y )= \widecheck y$.  
\end{lemma}

\begin{corollary}\label{cor:3}
	Suppose $G$ is a finite reflection group and for $y, x \in \R^n$ have 
	\[
	\langle \widecheck y  , \widecheck x  \rangle = \langle  y  ,   x  \rangle
	\]
	then there exists $h \in G$ such that $\widecheck x = h x$ and $\widecheck y = h y$. 
\end{corollary}

\begin{proof}
	Take $q , p \in G$ such that $q \widecheck y = y$ and $p \widecheck x = x$. Then 
	\[
	\langle p^{-1} q \widecheck y , \widecheck x \rangle = \langle \widecheck y , \widecheck x \rangle.
	\]
	By Lemma \ref{lem:3} we have the existence of $w \in C_{G} (\widecheck x)$ such that 
	$w p^{-1} q \widecheck y = \widecheck y$ and hence we may take $h := w p^{-1}$. 
\end{proof}

Let $C$ be a nonempty closed set in $\R^n$. Then the set valued projection operator $\Proj_C(\cdot)$ is well defined,
$$
\Proj_C(x) = \Argmin_{y\in C} \|x-y\|_2.
$$
By $\dist(x,C)$ denote the relevant distance from a point $x\in \R^n$ to $C$,
$$
\dist(x,C) = \min_{y\in C} \|x-y\|_2.
$$

\begin{proposition}\label{lem:ProjAction}
	Let $C$ be a nonempty closed set invariant under a finite reflection group $G$. Then for every $g\in G$ we have $g\Proj_C(x) = \Proj_C(gx)$, where 
\end{proposition}
\begin{proof}
First of all, observe that due to symmetry $\dist(x,C)= \dist(gx,C)$ for all $g\in G$. Let $y\in \Proj_C(x)$. Then $\|gy - gx\| = \|y-x\|$, hence, $g y \in \Proj_C(gx)$. The reverse inclusion is obtained by considering $x' = g^{-1} (gx)$.	
\end{proof}

We next summarise the properties of projection onto an invariant set.

\begin{theorem}\label{thm:encompass} Let $C\subset \R^n$ be a nonempty closed set invariant under the action of a finite reflection group $G$. Then the following statements are true:
	\begin{itemize}
		\item[(i)] for $x\in \mathcal C$, we have $\Proj_C(x)=C_G(x)[\Proj_C(x)\cap \mathcal C]\neq \emptyset$; 
		\item[(ii)] $y\in \Proj_C(x)$ if and only if $\widecheck y \in \Proj_C(\widecheck x)$ and $\langle \widecheck x,\widecheck y\rangle = \langle x,y\rangle$;
		\item[(iii)] if $C$ is convex, and $x\in \mathcal C$, then $\Proj_C(x) \in C\cap \mathcal C$. 
	\end{itemize}
\end{theorem}
\begin{proof} First of all, observe that the set-valued projection operator ${\Proj}_C$ is nonempty at $x$. 
	
Let $y'$ be an arbitrary point in ${\Proj}_C(x)$, and let $g\in G$ be such that $y' = g\widecheck y'$. Then by Lemma~\ref{lem:3} (and keeping in mind that $\widecheck x = x$),
$$
\langle y', x\rangle = \langle g \widecheck y', \widecheck x\rangle \leq  \langle \widecheck y', \widecheck x\rangle  = \langle \widecheck y', x \rangle
$$
hence, $\|\widecheck y' - x\| \leq \|y'-x\|$. Therefore,  $\widecheck y' \in {\Proj}_C(x)\cap \mathcal C$. Now if we assume that $y' \notin C_G(x)[{\Proj}_C(x)\cap \mathcal C]$ then there cannot exists a 
	$q \in C_G(x)$ such that $y' = q \widecheck y'$. Lemma~\ref{lem:3} now implies that 
	the inequality $ \langle g \widecheck y', \widecheck x\rangle < \langle \widecheck y', \widecheck x\rangle$ is strict and hence $\|\widecheck y' - x\| < \|y'-x\|$
	contradicting $y' \in {\Proj}_C(x)$. Thus  $C_G(x)({\Proj}_C(x) \cap \fcham ) \supseteq {\Proj}_C(x)$. 
	For the reverse inclusion, take $\widecheck y \in {\Proj}_C (x)$ and note 
	for all $p\in C_G(x)$, $p^{-1}\in C_G(x)$ and $p^{-1}p \widecheck y = \widecheck y$, $p^{-1}\widecheck x = \widecheck x$, hence, by Lemma~\ref{lem:3} we have $\langle p\widecheck y , x\rangle = \langle \widecheck y, x\rangle$, so 
	$$
	\|\widecheck y - x\| = \|p\widecheck y - x\|, 
	$$
	and so $p\widecheck y \in {\Proj}_C(x)$. Hence, we have (i).

From the convexity and uniqueness of projection we have (iii). 
	
	It remains to show (ii). Let $y\in {\Proj}_C(x)$, and find $g\in G$ such that $g x = \widecheck x$. Then 
	\begin{equation}\label{eq:0034}
	\|y - x\| = \|g y - g x\| = \|gy - \widecheck x\|.
	\end{equation}
	At the same time,
	\begin{equation}\label{eq:0035}
	\min_{z\in C}\|z - x\| = \min_{z\in C}\|gz - gx\| =  \|gy - \widecheck x\|.
	\end{equation}
	From \eqref{eq:0034} and \eqref{eq:0035} and the invariance of $C$ we have $gy \in {\Proj}_C(\widecheck x)$. By (i) there exists $h\in C_G(x)$ and $\tilde y \in {\Proj}_C(x)\cap \mathcal C$ such that $y = g^{-1} h \tilde y$, hence by uniqueness, $\tilde y = \widecheck y$. Thus 
$$
\langle y , x \rangle = \langle g y , \widecheck x \rangle = \langle h \widecheck y , \widecheck x \rangle =  \langle \widecheck y , h^{-1}\widecheck x \rangle
	=  \langle  \widecheck y , \widecheck x \rangle
$$  
and we proved the necessary part of (ii): if $y\in {\Proj}_C(x)$, then $\widecheck y \in {\Proj}_C(\widecheck x)$ and $\langle \widecheck x, \widecheck y\rangle = \langle x,y\rangle$. 
	
The sufficient direction of (ii) is trivial: assume that $\widecheck y \in {\Proj}_C(\widecheck x)$ and $\langle \widecheck x, \widecheck y\rangle = \langle x,y\rangle$. It is not difficult to observe that $d(x,C) =d(\widecheck x,C)$ (where by $d$ we denote the distance). Moreover, from $\langle x,y\rangle  = \langle \widecheck x, \widecheck y\rangle$ we have $\|x-y\| = \|\widecheck x - \widecheck y\|$, hence, $y\in \Proj_C(x)$.    
\end{proof}

We immediately get the following corollary for the case when the projection is a singleton.

\begin{corollary} \label{cor:uniqueproj}
	Let $C \subseteq \R^n $ be a closed convex set invariant under a finite reflection group $G$, then
	\begin{equation}
	\{y\} = {\Proj}_{C} (x) \quad \iff \quad \widecheck y =   {\Proj}_{C} 
	(\widecheck x)
	\end{equation}
	moreover we have $\langle \widecheck x , \widecheck y \rangle = \langle x , y \rangle$. 
\end{corollary}

\section{Proximal normals and subdifferentials}\label{sec:normals}

Some of the relations obtained in the previous sections for the projections on invariant sets can be used to study proximal normals to sets which are not necessarily convex. This in turn can be applied to epigraphs to deduce results about subdifferentials. 

\begin{definition}[Proximal normal cone] Given any set $C\subset \R^n$, the proximal normal cone to $C$ at $x\in S$ is the cone 
	$$
	\hat N_C(x) = \{\lambda(y-x):\; \lambda \geq 0 , \; x \in {\Proj}_C(y) \}.
	$$
\end{definition}
The proximal cone consists of all points along which we project on a given point $x\in C$. We have $y \in \hat{N}_{C} (x)$ iff
there exists $\alpha >0$ such that we have $x \in \Proj_{C} (x + \alpha y)$ (in which case 
for any smaller $\alpha$ the same inclusion holds). 

The following theorem gives a relation between the proximal normals at an arbitrary point of the set $C $ and their counterparts inside the fundamental chamber. This result holds when $C$ satisfies the condition that convex hulls of all orbits of points in $C$ belong to $C$. Observe that this condition is significantly more general than convexity. The shapes in Fig.~\ref{fig:dumbbell} 
\begin{figure}[ht]
	\centering
	\includegraphics[scale=0.5]{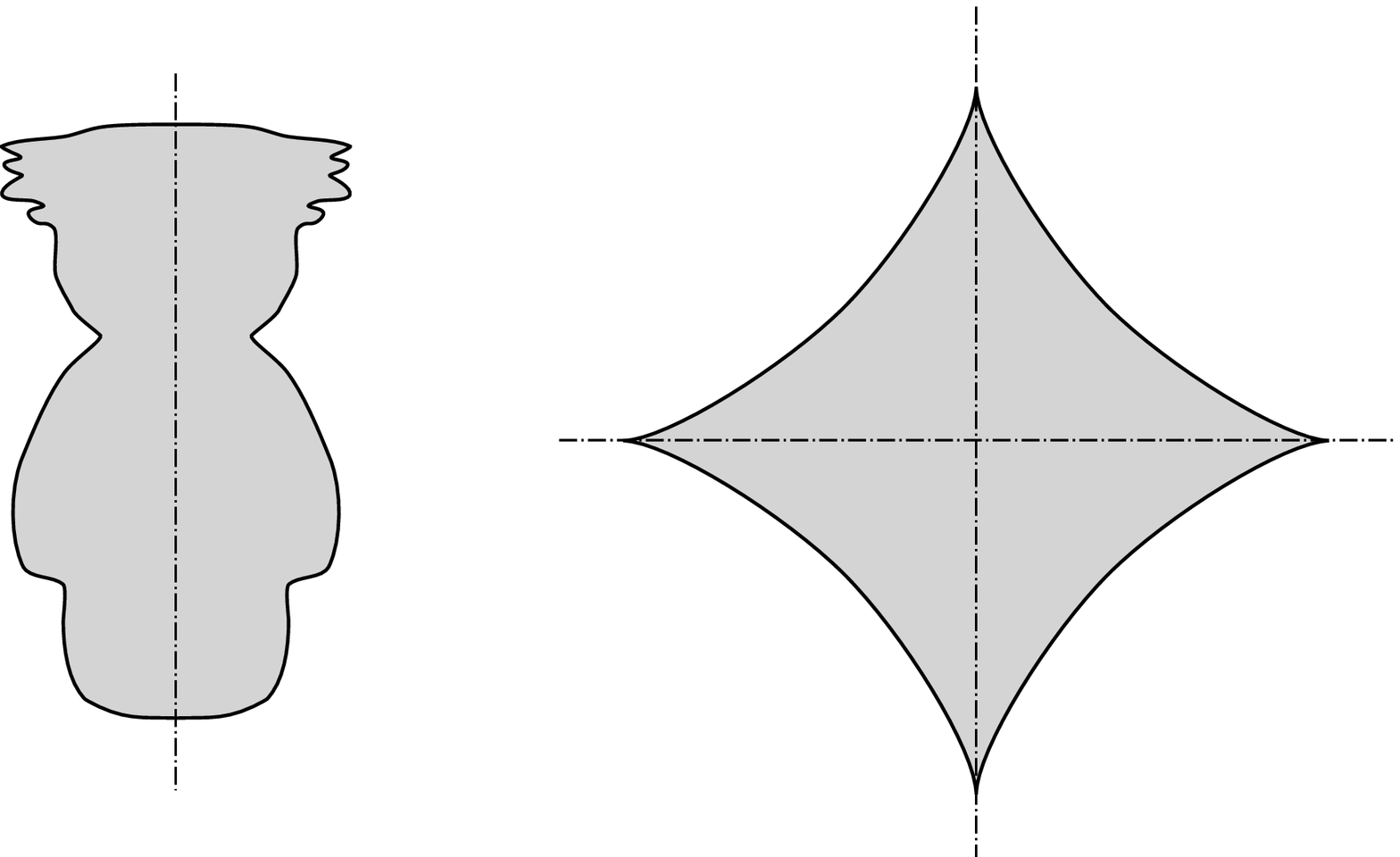}
	\includegraphics[scale=0.16]{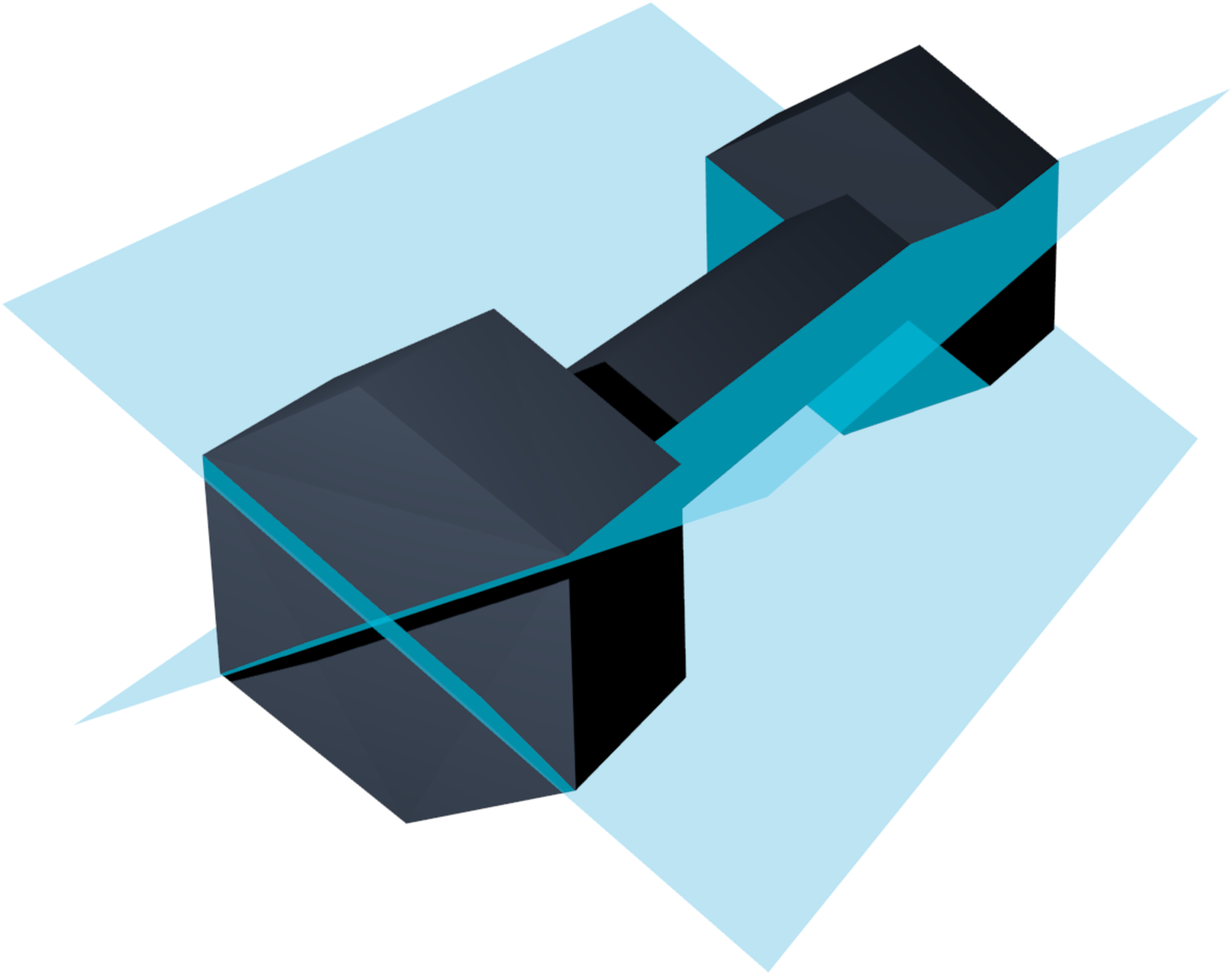}
	\caption{Sets that satisfy property \eqref{eq:propertyA} with the reflection group generated by mirrors shown as the dash dotted lines and the semi transparent planes}
		\label{fig:dumbbell}
\end{figure}
are examples of sets that are not convex, but satisfy the conditions of the next theorem.
\begin{theorem}\label{thm:3}
	Suppose $G$ is a finite reflection group and $C$ is a $G$-invariant set with the following property
\begin{equation}\label{eq:propertyA}
\forall x \in \bd C \text{ we have }  \co \orbit (x) \subseteq C.
\end{equation}	
	Then $y \in \hat{N}_{C} (x)$ iff  $\widecheck y \in \hat{N}_{C} (\widecheck x )$
	and $\langle \widecheck x , \widecheck y \rangle = \langle x , y \rangle$. 
\end{theorem}

\begin{proof}
	We have $y \in \hat{N}_{C} (x)$ iff $x \in {\Proj}_{C} ( x + \alpha y )$ for all sufficiently small $\alpha >0$. Applying Theorem \ref{thm:encompass} (ii) we have $\widecheck x  \in {\Proj}_{C} ( \widecheck{x + \alpha y})$ and $\langle \widecheck x , \widecheck{x + \alpha y} \rangle = \langle x, x + \alpha y  \rangle$. Thus by Corollary~\ref{cor:3} there exists $g \in G$ 
	with $ g x = \widecheck x$ and $g (x + \alpha y ) = \widecheck{x + \alpha y }$. 
	Thus $\widecheck{x + \alpha y } = \widecheck x + \alpha g y$ and so
	$\widecheck x \in {\Proj}_{C} ( \widecheck x + \alpha  g y )$. It remains to show that 
	there exists $h \in C_{G} (\widecheck x )$ with $h g y = \widecheck y$ since then
	it would follow that  $\widecheck x = h \widecheck x   \in {\Proj}_{C} ( h \widecheck x + \alpha h g y) =  {\Proj}_{C} (\widecheck x + \alpha  \widecheck y ) $ (thus $\widecheck y \in \hat N_C(x)$)
	and $\langle \widecheck x , \widecheck x + \alpha g y \rangle = \langle x, x + \alpha y  \rangle$ renders $\langle \widecheck x , \widecheck y \rangle = \langle x, y \rangle$. 
	
	To this end we recall that $G$ is a group of isometries and that there exists $k \in G$
	such that $ky = \widecheck{y}$. Now condition \eqref{eq:propertyA} and the fact that $\widecheck{x} 
	\in \bd C$ implies $\widecheck{x} \in \bd \co
	\orbit(\widecheck{x})$ and so $gy \in {N}_{\co \orbit(\widecheck{x})} (\widecheck{x})$ (the normal cone of convex analysis). Hence for all $z \in \co
	\orbit(\widecheck{x})$ we have 
	\begin{eqnarray*}
		\langle g y , z - \widecheck{x} \rangle & \leq & 0 \text{ and so } \\
		\langle  y , z \rangle = \langle gy ,g z \rangle &\leq & \langle gy , \widecheck{x} \rangle 
		\text{ for all $z \in \orbit( \widecheck{x})$ } \\
		\text{hence } \langle \widecheck{y} , z \rangle  = \langle y , k^{-1} z \rangle & \leq & \langle gy , \widecheck{x} \rangle \text{ for all $z \in \orbit(\widecheck{x})$}  \\
		\text{and so } \langle \widecheck{y} , \widecheck{x} \rangle & \leq & 
		\langle g y , \widecheck{x} \rangle \text{ as $\widecheck{x} \in \orbit( \widecheck{x})$}. 
	\end{eqnarray*}
	By Lemma~\ref{lem:3} we have 
	the existence of $h \in C_{G} (\widecheck x )$ such that $h g k  \widecheck y= \widecheck y$ or $ h g y = \widecheck y$ as required. 
	
	Conversely when $\langle \widecheck x , \widecheck y \rangle = \langle x , y \rangle$ 
	we may apply Corollary~\ref{cor:3} to obtain $h \in G$ such that $h \widecheck x =  x$
	and  $h \widecheck y =  y$, it then follows from  $\widecheck y \in \hat{N}_{C} (\widecheck x )$ that $y \in \hat{N}_{C} ( x )$.
\end{proof}

When $f$ is $G$-invariant we have $f(g x) = f(x)$ for all $g \in G$ and so the 
we can extend the application of $G$ to $X$ to $X \times \R$
and the $\epi f$  that is consistent
with the invariance of the function $f$. Namely, for $g \in G$ we can extend $G$ 
to $X \times \R$ via 
\[
g (x, \alpha) = (gx ,\alpha) . 
\]
This extended transformation is also an isometry as 
\begin{equation}
\label{eq:exttransform}
\langle g (x, \alpha) , g (y, \beta) \rangle 
= \langle g (x),  g (y) \rangle + \alpha \beta = \langle x,  y \rangle + \alpha \beta 
= \langle  (x, \alpha) , (y, \beta) \rangle .
\end{equation}

\begin{lemma}
	Suppose $G$ is a finite reflection group and $f$ is a $G$ invariant function then 
	$\operatorname{epi} f$ is a $G$ invariant set with respect to the extended 
	group transformation.
\end{lemma}

\begin{proof}
	
	Take any $(x,y)\in \epi f$, then $g(x,y) = (gx, y)$, 
	where $y\geq f(x) = f(gx)$, hence, $g(x,y) \in \epi f$, and $g \epi f \subset \epi f$. This also yields the reverse inclusion, as $G$ is a group.
\end{proof}

We are now able to extend Theorem~\ref{thm:main:1} to non-convex function and
the proximal subdifferential. The proximal subdifferential can be generated via normal vectors as a slice of the proximal normal cone to the epigraph of the function, given by
\[
\hat{N}_{\operatorname{epi}f} (x, f(x)) := \{ (y , \beta ) \mid \exists \alpha >0 
\text{ such that }  (x, f(x)) \in {\Proj}_{\operatorname{epi}f} ((x, f(x))  + \alpha (y, \beta) \}. 
\]
The proximal subdifferential (resp. singular proximal subdifferential) is given by 
\begin{eqnarray*}
	\partial_p f (x) &:= & \{ y \mid (y , -1 ) \in \hat{N}_{\operatorname{epi}f} (x, f(x))  \} \\
	\text{resp. }  \partial^{\infty}_p f (x) &:= & \{ y \mid (y , 0 ) \in \hat{N}_{\operatorname{epi}f} (x, f(x))  \} .
\end{eqnarray*}

\begin{theorem}\label{thm:proximalsub} Let $f:\R^n\to \R_{+\infty}$ be invariant under a finite reflection group $G$ and assume that $\epi f $ satisfies property \eqref{eq:propertyA} with respect to the extended transformation \eqref{eq:exttransform}. Then $y\in \partial_p f(x)$ if and only if 
	\begin{equation}\label{eq:mainprox}
	\widecheck  y \in \partial_p f(\widecheck  x) \quad \text{and} \quad \langle \widecheck  x, \widecheck  y\rangle = \langle x,y\rangle.
	\end{equation}
Furthermore $y\in \partial^{\infty}_p f(x)$ if and only if 
	\begin{equation}
	\widecheck  y \in \partial^{\infty}_p f(\widecheck  x) \quad \text{and} \quad \langle \widecheck  x, \widecheck  y\rangle = \langle x,y\rangle.
	\end{equation}
\end{theorem}

\begin{proof}
	We have $ y \in \partial_p f (x) $ iff 
	$(y, -1 ) \in \hat{N}_{\epi f} (x, f(x))$. Applying Theorem~\ref{thm:3} 
	we have this true iff $(\widecheck y, -1 ) \in \hat{N}_{\operatorname{epi}f} (\widecheck x, f(\widecheck x))$ and $\langle (\widecheck y, -1 ) , (\widecheck x, f(\widecheck x)) \rangle 
	=\langle ( y, -1 ) , (  x, f(  x)) \rangle  $. The former identity is equivalent to 
	$\widecheck y \in \partial_p f (\widecheck x)$ and the second equivalent to 
	$\langle \widecheck  x, \widecheck  y\rangle = \langle x,y\rangle$. The second assertion follows similarly. 
\end{proof}

\begin{remark}
	The classic counter example to such a result holding for arbitrary group invariant function
	is the function $f(x_1, x_2, \dots, x_n)=x_1\times x_2\dots\times x_n$ under the group of permutations
	${\bf P} (n)$. 
\end{remark}

The set of functions that satisfy the conditions of Theorem~\ref{thm:proximalsub} is significantly broader than convex, and in particular includes Schur convex functions defined with respect to the preordering induced by group majorisation. 

\subsection{Schur convex functions}\label{sec:Schur}

Recall that a function $f$ is called Schur convex with respect to a preordering $\succeq$  on $\R^n$ if it is isotone with respect to this preordering: 
$$
f(x) \geq f(w) \qquad \text{ whenever } \qquad x \succeq w.
$$
Of particular interest to us is the partial order on the fundamental chamber $\fcham$  induced by {\em group majorisation} \cite{NiezgodaSchur,Niezgoda1998}.

\begin{definition}
	The group majorisation w.r.t. the finite reflection group $G$ is the preordering on $
	\R^n$ defined by 
	$$
	 x \succeq_G y \quad \text{ iff } y \in \co \orbit(x),
	$$ 
	where $\co O(x)$ is the convex hull of the orbit of $x$ under the action of $G$.
\end{definition}

Schur convex functions are a broader class than convex functions, and we show in what follows that all $G$-invariant pseudo-convex functions are Schur convex.

It is not difficult to observe that Schur convex functions defined via the $G$-preordering have the property that 
$$
\co \orbit (x) \subseteq \operatorname{lev}_{\alpha} f :=
\{ y \mid f(y) \leq \alpha \}\quad \text{ for all } \quad \alpha \geq f(x),
$$
and hence we can show that they satisfy the conditions of Theorem~\ref{thm:proximalsub}. 

\begin{lemma}
	Let $f:\R^n\to \R_{+\infty}$ be Schur convex and invariant under a finite reflection group $G$. Then the subset $\operatorname{epi} f \subseteq \R^{n+1}$ satisfies property \eqref{eq:propertyA} of
	Theorem~\ref{thm:3} with respect to the extended group action. 
\end{lemma}

\begin{proof}
	Take an arbitrary $(x, f(x)) \in \operatorname{bd} \operatorname{epi} f$ and consider the 
	$(y, \alpha) \in \co \orbit (x, f(x))$. Then there exists positive numbers
	$\lambda_i$ summing to unity and $g_i \in G$ for $i=1,\dots,k$ such that $(y, \alpha) 
	=\sum_{i=1}^k \lambda_i (g_i x, f(x))$. That is $y = \sum_{i=1}^k \lambda_i g_i x 
	\in \co \orbit (x)$ and $\alpha = f(x)$. As $f$ is Schur convex we have 
	$f(y) \leq f(x)=\alpha$  implying $(y , \alpha) \in 
	\operatorname{epi} f$. 
\end{proof}

Recall that pseudo-convex functions are defined as functions with convex level sets. We have the following result. 

\begin{proposition}
	Let $f: \R^n \to \R_{+\infty}$ be a $G$ invariant pseudo-convex function then $f$ is Schur convex with respect to the group preordering. 
\end{proposition}

\begin{proof}
	Let $x \succeq_{G} w $ then we must have $w \in \co \orbit (x)$ 
	so we may take positive scalars $\lambda_1, \dots, \lambda_k$ summing to unity and 
	$g_i \in G$ for $i=1,\dots, k$ such that $w = \sum_{i=1}^k \lambda_i g_i x$. 
	Now suppose $f(x) < f(w)$ then by taking $f(x) < \alpha < f(w)$ we have 
	$x \in \operatorname{lev}_{\alpha} f$ but $w \notin \operatorname{lev}_{\alpha} f$.
	But as $f(g_i x)  = f(x) < \alpha$ for all $i$ we have, via pseudo-convexity and the convexity of level sets, that $w \in \operatorname{lev}_{\alpha} f $, a contradiction. 
\end{proof}

\section{Permutation invariance and  compressed sensing}\label{sec:compressed}

In the area of sparse signal recovery and compressed sensing the goal is to reconstruct a signal with the minimum number of non zero elements in the solution vector.  This is measured 
by the $l_0$ norm that counts the number of non zero components of the vector. 
Very good solutions are obtained by instead adding a penalty term involving the 
$l_1$ norm, thus solving a discrete problem via a continuous approximation.
 The important role played by symmetries is with respect to  the sparsity constraint 
\begin{equation}\label{eq:sparsity}
 C_s := \{ x \mid \| x \|_0 \leq s\}.
\end{equation}
The sparsity constraint is invariant under the permutation of components of its vectorial elements. A basic compressed sensing problem formulation can be stated as follows:
\[
\min f(x)  \quad \text{ subject to } x \in C_s \cap B,
\]
where $f$ and  $B$ are permutation invariant, and $B$ is usually assumed to be convex. 
Many algorithms designed to solve this problem require the 
calculation of the Euclidean  projection 
$$
\| x - \Proj_{C_s \cap B} (x) \|_2 =  \min \{ \|x - y \|_2 \mid y \in C_s \cap B\}. 
$$
If this can be computed efficiently then one can use projected descent methods or similar optimisation techniques. The nonconvexity of $C_s$ means that the projection is not necessarily  unique.

Observe that the sparsity constraint $C_s$ is invariant (as a set in $\R^n$) with respect to the action of the coordinate permutation group $\Perm(n)$. Observe that $\Perm(n)$ is a finite reflection group that can be generated by mirrors that swap pairs of coordinates; we will always assume that its fundamental chamber corresponds to the nonincreasing reordering of coordinates.

In \cite{Beck2014} the authors define a type-2-symmetric set if it is not only permutation invariant but also invariant under sign changes. This type of symmetry is well known in group theory \cite{Niezgoda1998}. The fundamental chamber associated with this group corresponds to 
\begin{equation}\label{eq:typetwocham}
\fcham := \{x \in \R^n \mid x_1 \geq x_2 \geq \dots \geq x_n \geq0\}
\end{equation}
and $\widecheck x = |x|^{\downarrow}$. We denote this group by $\Perm_2$.

The intersection of $C_s$ with the fundamental chamber $\fcham$ of the type-2 symmetric group $\Perm_2$ is a convex set. In view of Theorem~\ref{thm:encompass} the study of projections on sparsity constraints hence reduces to the study of projections onto convex sets. We sharpen these statements in the following two results.

\begin{lemma}\label{lem:CsConvexIntersection} Let $s$ be a nonnegative integer, and let $\Perm_2$ be the type-2 symmetry group, with $\fcham$ being its fundamental chamber defined by \eqref{eq:typetwocham}. Let $B$ be a nonempty closed convex set invariant under $\Perm_2$. Then the intersection of $C_s\cap B$ (where $C_s$ is defined by \eqref{eq:sparsity}) with the fundamental chamber $\fcham$ is a convex set.
\end{lemma}
\begin{proof}
	Let $x,y\in \fcham \cap C_s \cap B$, and let $z = \alpha x+ (1-\alpha )y$ with $\alpha \in (0,1)$. From the definition \eqref{eq:typetwocham} of $\fcham$, the first $s$ coordinates of both $x$ and $y$ are nonnegative, and the rest are zero. It is not difficult to observe that the same is true for $z$, hence, $z\in C_s$. Moreover, 
	$$
	z_{i+1} = \alpha x_{i+1} +(1-\alpha)y_{i+1} \geq \alpha x_i +(1-\alpha)y_i =  z_i \quad \forall i\in \{1,\dots, n-1\},
	$$
	hence, $z_i\in \fcham$. Clearly $z\in B$ due to the convexity of $B$. By the arbitrariness of our choice of $z$ we have $\co (C_s\cap \fcham) = C_s\cap \fcham$.
\end{proof}
 
In view of Lemma~\ref{lem:CsConvexIntersection} computation of projections on the sparsity constraints reduces to computing projections on a linear subspace. More precisely, we have the following result.

\begin{lemma}\label{lem:specialprojection} Let $x\in \fcham$, where $\fcham$ is the fundamental chamber of $\Perm_2$. Then 
$\Proj_{C_s\cap \fcham}= \{y\}$, where
$$
y_i = 0 \quad \forall i>s;\quad y_i = x_i \quad \forall i\leq s.
$$	
\end{lemma}
\begin{proof}
First of all, observe that $y\in C_s\cap \fcham $. To prove that $y$ is indeed the projection of $x$ onto $S: = C_s\cap \fcham$, it is enough to show that 
$$
\langle x-y, z-y\rangle \leq 0 \qquad \forall\, z\in S.
$$
We have 
$$
(x-y)_i = 0 \quad \forall i\leq s; \qquad  (z-y)_i = 0\quad \forall i >s,
$$
hence,
$$
\langle x-y, z-y\rangle = 0 \qquad \forall\, z\in S.
$$
\end{proof}

\begin{lemma}\label{lem:specialsparcityprojection} Let $s$ be a nonnegative integer, and let $x\in \R^n$. Assume that $B$ is a nonempty closed convex set invariant under $\Perm_2$. Then the projection of $x$ onto $C_s\cap B$ can be computed as follows
$$
\Proj_{C_s\cap B}(x) = Q^{-1}C_{\Perm_2}(x) \Proj_{\fcham\cap C_s\cap B}(Q x), 
$$ 
where $Q\in \Perm_2$ is such that $Qx = |x|^\downarrow$, and $C_{\Perm_2}(x)$ is the stabiliser of $x$.
\end{lemma}
\begin{proof} Observe that  $\widecheck x = Q x$. By Lemma~\ref{lem:CsConvexIntersection} the intersection $C_s\cap \fcham \cap B$ is a convex set, so we use the notation $\{p\} = \Proj_{C_s\cap B \cap \fcham}(\widecheck x)$. Our goal is to show that $\{p\} = \Proj_{C_s\cap B}(\widecheck x)\cap \fcham$. Then the result follows directly from Theorem~\ref{thm:encompass}~(i) and Lemma~\ref{lem:ProjAction}. 
	
Observe that the inclusion $\Proj_{C_s\cap B }(\widecheck x) \cap\fcham \subset \Proj_{C_s \cap B\cap \fcham }(\widecheck x)$ is trivial, and in view of Lemma~\ref{lem:CsConvexIntersection} the set on the right hand side is a singleton, and we only need to show that $\Proj_{C_s \cap B}(\widecheck x)\cap \fcham $ is nonempty to demonstrate the reverse inclusion. Assume the contrary, then there exists a point $z$ in $\Proj_{C_s\cap B }(\widecheck x)$ such that $\|z-x\|<\||z|^\downarrow-x\|$. Observe that by our assumption $z\notin \fcham$, hence, $z':= |z|^\downarrow \neq z$, but $z'\in C_s \cap B\cap \fcham$. It is straightforward that since $x = |x|^\downarrow$,
$$
\langle |z|,x\rangle \geq \langle z,x\rangle;\qquad \langle |z|^\downarrow,x\rangle \geq \langle |z|,x\rangle,
$$
hence, 
$$
\|z-x\| \geq \||z|^\downarrow - x\|, 
$$
and hence our assumption is wrong.
\end{proof}

The major results of \cite[Section~4]{Beck2014} deal with the projections onto the intersections of $G$-invariant sets and sparsity constraints. Observe that in the case when $B$ is a close convex set invariant with respect to $\Perm_2$, we have obtained the explicit representation of the projection onto $C_s\cap B$.

\section*{Acknowledgements} We are grateful to our colleagues Bill Moran and Lawrence Reeves for their patient explanations on Coxeter groups and representation theory.

\bibliographystyle{plain}
\bibliography{references}

\end{document}